\pdfoutput=1
\documentclass[10pt,leqno]{amsart}

\usepackage{graphicx}
\usepackage{indentfirst,csquotes}

\usepackage{amssymb,amsthm,amsmath}
\usepackage{amsfonts}
\usepackage{booktabs}
\usepackage{bussproofs}
\usepackage{tikz}
\usetikzlibrary{arrows.meta}
\usepackage{xcolor,paralist,hyperref,titlesec,fancyhdr,etoolbox}
\usepackage{orcidlink}

\providecommand{\doi}[1]{\href{https://doi.org/#1}{doi:#1}}

\newtheorem{thm}{Theorem}[section]

\newtheorem{prop}{Proposition}[section]

\theoremstyle{definition}
\newtheorem{defi}{Definition}[section]

\theoremstyle{remark}
\newtheorem{rem}{Remark}[section]

\makeatletter
\providecommand{\@secnumpunct}{.}   % restored: titlesec bypasses the amsart definition
\makeatother

\titleformat{\section}{\normalfont\large\bfseries\centering}{\thesection.}{0.5em}{}
\titlespacing*{\section}{0pt}{2.6ex plus 1ex minus .2ex}{1.6ex plus .2ex}

\hypersetup{ colorlinks=true, linkcolor=black, filecolor=black, urlcolor=black,
             citecolor=black }

\begin{document}

\title[Human and AI-generated texts between modal logic and statistics]{Human and AI-generated texts between modal logic and statistics}

\author[S.\ Cuconato]{Simone Cuconato\,\orcidlink{0000-0003-0277-9575}}
\address{Department of Physics, University of Calabria, Italy}
\email{simone.cuconato@unical.it}

\author[D.\ Ferrari]{Donato Ferrari\,\orcidlink{0009-0006-6583-1341}}
\address{Department of Business and Legal Sciences, University of Calabria, Italy}
\email{donato.ferrari@unical.it}



\date{}

\begin{abstract}
We read the geometry of semantic neighbourhood graphs as modal logic and give that reading a statistical form, in order to make precise the structural difference between human and machine-generated text. Texts are the worlds of a finite frame whose accessibility is the $k$-nearest-neighbour relation of a transformer embedding, and the symmetry, transitivity, Euclideanity and seriality frequencies of the two subcorpora are shown to be \emph{degrees of validation} of the modal axioms $\mathsf{B},\mathsf{4},\mathsf{5},\mathsf{D}$. Each degree is at once the proportion of instances of a rule that the subframe licenses in Negri's labelled calculus $\mathsf{G3.K}$ and a plug-in estimate of a population probability. A prompt-balanced comparison finds consistently higher artificial degrees for $\mathsf{4}$ and $\mathsf{5}$. We add a degree of groundedness and of situatedness, and recast the licensing reading in Cuconato's one-sided sequent-style tableaux, where each degree becomes a rate of set membership.
\end{abstract}

\maketitle

\section{Introduction}\label{sec:intro}

Among the digital instruments now in routine scholarly use, artificial intelligence \cite{ref_scarcello} occupies a peculiar position \cite{cuconatoscarcellobeneduci}. It analyses large corpora, recognises stylistic regularities through stylometry, and reads the affective layering of a text through sentiment analysis; in doing so it supports the digital humanities \cite{ref_cuconato2024fol,ref_cuconato2026soundness} while quietly unsettling the assumptions on which they rest -- the nature of a text, and the conditions of its intelligibility \cite{ref_cuconato2024modal}. The present study grows from a deliberately simple question: can \emph{perceived space} serve as a criterion for telling human writing apart from synthetic writing? Posed in this form the question is no longer merely linguistic but geographical, where geography is understood not as the inventory of physical locations but as the study of the ways in which human beings perceive, organise, remember, and narrate space. So construed, geography is not a frame imposed on textual analysis from without; it is the vantage from which the relation between language, experience, and situated knowledge first becomes visible.

The encounter between geography and computation is not new \cite{ref_torrens2018artificial}. Computational tools entered geographical research first as a means of managing large environmental datasets and then as instruments for modelling and visualising spatial phenomena, until artificial intelligence ceased to be an external device and became part of the very infrastructure through which several disciplines observe, classify, and describe the world \cite{ref_zhao2021deep}. Its growing fluency, however, makes its limits more conspicuous, not less. A model writes coherent prose and will describe a journey, a city, a childhood memory, a coastline with unimpeachable syntax; what remains uncertain is whether such descriptions carry the epistemic depth of those produced by a subject who has actually lived, perceived, and remembered the world from within it. This is the crux of the debate over meaning and form: fluency in the manipulation of form is not, by itself, access to the situated meaning that form conveys in human hands. Human beings do not merely report events -- they place them; they remember from somewhere and speak from somewhere, organising meaning through the implicit geography of lived experience \cite{ref_tuan1977space}, a conviction that animates the recent turn to narrative method in geographical research.

The corpus behind this paper was designed to bring that difference to the surface \cite{ref_depascale2026geospatial}. A set of Italian texts written by human participants is compared with synthetic texts generated from the same thematic nuclei: the human narratives were gathered through a questionnaire that asked participants to recount a personally significant event without ever requiring them to name a place, broad thematic cues were then extracted, and a chatbot was prompted with the same cues to produce comparable artificial texts at a fixed ratio of one human sample to three generated ones \cite{ref_cuconatoferrari}. The design lets one ask not whether the two kinds of writing differ in style -- they do, trivially -- but whether they differ in the way they \emph{build a world}: in how references to place interact with the verbs of perception, movement, and memory. The working hypothesis is that human writing retains a geographical density irreducible to linguistic probability, and it is best stated in the idiom of humanistic geography, for which place is never a neutral coordinate but space transformed by memory, perception, and meaning \cite{ref_levy1999tournant}.

This qualitative thesis has since been given quantitative body. Representing the corpus as a graph in which two texts are joined when their embeddings are near, one can measure how the resulting semantic space is locally organised, and find that the human and artificial subcorpora organise it differently \cite{ref_cuconato2025situation}. The difficulty lies not with the data but with the language used to read them. That language speaks of accessibility, of neighbourhoods that are or are not closed, of interiors and boundaries -- the vocabulary of modal logic and of the finite topology it induces -- and yet the connection is left at the level of metaphor: the frequencies of symmetry or transitivity in a neighbourhood graph are reported as geometric facts, and the modal words are borrowed to gloss them. Our aim is to remove the metaphor and put a calculus where it stood, and then to put a statistic where the calculus stands. We argue that those frequencies are \emph{degrees of validation} of the normal modal axioms; that, read proof-theoretically, each is the fraction of the applications of a structural rule that the corpus is entitled to make; and that, read statistically, each is an estimator of a population parameter, so that the separation between the two corpora becomes a hypothesis to be tested rather than a contrast to be admired.

The instrument for the first reading is the proof theory of modal logic in the labelled style of Negri \cite{ref_negri2005}, as systematised in Negri and von Plato \cite{ref_negrivonplato2011}. Its decisive feature is that frame conditions are not external constraints imposed on a semantics but \emph{rules} of a sequent calculus, added to a common core $\mathsf{G3.K}$ without disturbing its structural behaviour -- the admissibility of contraction and of cut, and the invertibility of the rules. Once symmetry, transitivity, Euclideanity and seriality are written as the rules $\mathrm{Sym},\mathrm{Trans},\mathrm{Eucl},\mathrm{Ser}$, the empirical question ``how symmetric is the human subframe?'' acquires a proof-theoretic translation: ``what proportion of the instances of $\mathrm{Sym}$ does the human subframe license?'' The instrument for the second reading is the elementary statistical observation that a ratio of successes to eligible instances estimates a population validation probability. Its force here is that it turns the proof-theoretic licensing density into a quantity whose empirical uncertainty can be evaluated through dependence-aware resampling while respecting the graph structure induced by semantic neighbourhoods.

The paper is organised so as to make this double passage visible step by step. Section~\ref{sec:background} recalls the situation-theoretic background and the semantic role of reference, and equips the situated consequence with a \emph{degree of groundedness} and a \emph{degree of situatedness}. Section~\ref{sec:empirical} fixes the empirical setting, reproduces the relational frequencies, and records the global geometric statistics from which the frame is built. Section~\ref{sec:semantics} gives the modal semantics of the frames, shows that the situation-based system is, as ordinarily stated, inert, recalls the correspondence between the five axioms and their frame conditions, and identifies the boundary structure of the frame with the Shannon entropy of the class label across a neighbourhood. Section~\ref{sec:degrees} is the conceptual core: it defines the degree of validation, identifies it with the measured frequencies, reads it both as a licensing density and as an empirical plug-in estimator, and develops a dependence-aware comparison between the human and artificial corpora through prompt-balanced resampling. Section~\ref{sec:calculus} presents $\mathsf{G3.K}$ in Negri's style, establishes the structural properties that make the licensing reading legitimate, shows on the derivation of factivity how a formal proof inherits its empirical weight from the frequencies, and recasts that reading in the one-sided sequent-style tableaux of Cuconato, where a licensed step is confirmed by a single test of set membership. A discussion of scope and limits closes the paper.

\section{The situation-theoretic background}\label{sec:background}

Writing is never a merely formal arrangement of linguistic units, and nowhere is this assumption more consequential than in the comparison between human and synthetic text. Its pivot is the distinction between space and place. Space is abstract extension, a field of possible positions and relations; place is space transformed by meaning --- not only where something happens but where it becomes intelligible to a subject. A house, a street, a coastline, a room acquires narrative force because it is bound to memory, affection, loss, or expectation, and in human writing such references seldom serve as mere background: they act as cognitive and emotional anchors that let the narrated event take shape within a recognisable world. This is the lesson of humanistic geography, which treats place not as a neutral coordinate but as a lived and interpreted dimension of existence \cite{ref_tuan1977space}.

What we shall call \emph{spatial anchoring} is the capacity of a text to bind an event to a meaningful place and, through that binding, to produce a denser intelligibility; in human narratives place often works as a threshold between outer reality and inner life, explicit in the naming of a city, implicit in the description of an interior, a path, a distance. Synthetic narratives challenge this model precisely because they reproduce its surface so faithfully --- the vocabulary of memory, the syntax of emotion, the conventional imagery of place --- while the function of those elements changes, for the generated text issues not from a subject who remembers and inhabits but from a system that computes linguistic probability. Spatiality, there, is simulated rather than lived, assembled rather than recalled, and plausibility is not situatedness. To render this difference tractable rather than evocative we give it a formal carrier, the notion of situation space, in which meaning is anchored in structured configurations of experience rather than in isolated verbal units \cite{ref_giordani2024situation}.

A sentence, on this view, is not merely true or false at a world; it is \emph{about} something, and what it is about is a situation, a delimited portion of the world to which the sentence points and of which it speaks \cite{ref_fine2017theoryI,ref_fine2017theoryII}. Following Giordani \cite{ref_giordani2024situation}, we take the situations available at a world to be organised algebraically \cite{ref_barwise1989situation}.

\begin{defi}\label{def:sitspace}
A \emph{situation space} is a join-semilattice with greatest element, that is, a structure $(\mathit{Sit},1,\sqcup)$ in which $\sqcup$ is idempotent, commutative and associative, the induced order is $s\le s'$ iff $s\sqcup s'=s'$, and $1$ satisfies $s\sqcup 1=1$ for every $s$.
\end{defi}

The order records mereological inclusion of situations, the join is the smallest situation containing two given ones, and the top $1$ is the undifferentiated total situation. The reading we shall exploit is elementary but consequential: a text that anchors its assertions to specific, distinct places populates the lower, articulated part of the lattice, while a text that asserts without locating drifts towards the top, where all distinctions are washed out. The reference function makes this precise. Over a propositional language $\mathcal{L}$ with atoms $\mathcal{P}$, built by the classical connectives together with a constant $\top$, one assigns to each formula, at each world, the situation it concerns.

\begin{defi}\label{def:reference}
Let $D$ be a non-empty set of worlds. A \emph{situation frame} is a triple $(D,R,S)$ with $R:D\to\wp(D)$ an accessibility assignment, not required to be reflexive, and $S$ assigning a situation space $S(d)$ to each world. A \emph{reference function} $\sigma$ assigns to each formula at each world a situation, with $\sigma_d(p)\in\mathit{Sit}_d$ for atoms, $\sigma_d(\phi)=1_d$ whenever $\top$ occurs in $\phi$, and otherwise $\sigma_d(\phi)=\bigsqcup\{\sigma_d(p):p\in\mathcal{P}(\phi)\}$, the join of the referents of the atoms occurring in $\phi$.
\end{defi}

Reflexivity is deliberately omitted from the definition: neither the reference function nor the situated consequence introduced below makes any use of a self-accessibility $d\in R(d)$, and its omission keeps the situation frame in step with the self-loop-free $k$-nearest-neighbour relation that realises accessibility in Section~\ref{sec:empirical}. Truth at a world is classical --- $\top$ holds everywhere, an atom holds where the valuation $V$ makes it true, and the connectives behave Booleanly --- and ordinary consequence $\Delta\Vdash\phi$ is truth preservation at every world of every model. What the reference function adds is a second, finer relation, in which a conclusion may follow from premises only if it does not stray beyond what the premises are about.

\begin{defi}\label{def:situated}
The \emph{situated consequence} $\Delta\Vdash_\sigma\phi$ holds when $\Delta\Vdash\phi$ and, in every model and at every world $d$, $\sigma_d(\phi)\le\bigsqcup_{\delta\in\Delta}\sigma_d(\delta)$.
\end{defi}

The two relations come apart exactly where aboutness does work. One has $p\Vdash p\vee q$ unconditionally, yet $p\Vdash_\sigma p\vee q$ fails as soon as $\sigma_d(q)\not\le\sigma_d(p)$, since by Definition~\ref{def:reference} the disjunction inherits the referent $\sigma_d(p)\sqcup\sigma_d(q)$, which reaches into a region of the world that the premise $p$ never mentions. This is the formal home of the human/artificial contrast, and it turns on a single feature of the reference function at a world, namely whether any atom is anchored there at all. Assume the atom set $\mathcal{P}$ finite, as it is for any fixed vocabulary of markers, and write $\mathcal{P}_d=\{p\in\mathcal{P}:\sigma_d(p)\neq 1_d\}$ for the atoms with a proper referent at $d$. Call $d$ \emph{placeless} when $\mathcal{P}_d=\varnothing$ and \emph{place-grounded} otherwise.

It is tempting to say that the situated relation collapses onto ordinary consequence as soon as a world ceases to separate its markers, but the collapse is governed by placelessness alone, and the following proposition draws the boundary exactly.

\begin{prop}\label{prop:collapse}
At a world $d$ the side condition of Definition~\ref{def:situated} holds for every non-empty $\Delta$ and every $\phi$ with $\Delta\Vdash\phi$ --- so that $\Vdash_\sigma$ imposes nothing beyond $\Vdash$ at $d$ --- if and only if $d$ is placeless.
\end{prop}

\textbf{Proof.} Suppose $\mathcal{P}_d=\varnothing$. Then $\sigma_d(p)=1_d$ for every atom, and since $\sigma_d(\top)=1_d$ and the referent of a compound formula is the join of the referents of the atoms occurring in it, $\sigma_d(\phi)=1_d$ for every $\phi$; for non-empty $\Delta$ one has $\bigsqcup_{\delta\in\Delta}\sigma_d(\delta)=1_d\ge\sigma_d(\phi)$, and the side condition holds. Suppose instead $\mathcal{P}_d\neq\varnothing$, and fix $p\in\mathcal{P}_d$, so that $\sigma_d(p)\neq 1_d$. Taking $\Delta=\{p\}$ and $\phi=\top$ gives $p\Vdash\top$, whereas $\sigma_d(\top)=1_d\not\le\sigma_d(p)$, since $1_d\le\sigma_d(p)$ would force $\sigma_d(p)=1_d$ against the choice of $p$; the side condition fails. $\blacksquare$

Proposition~\ref{prop:collapse} fixes the exact threshold at which aboutness stops constraining inference: $\Vdash_\sigma$ coincides with $\Vdash$ at $d$ precisely when $d$ is placeless, and stays strictly finer at every place-grounded world. The geographical thesis --- that the artificial text writes from nowhere --- is, in this language, the claim that its models concentrate on placeless worlds, where the situational apparatus is inert.

The dichotomy is bivalent, but the data are continuous, and the apparatus should be too. We attach to each world two degrees, in exact parallel with the degrees of validation of Section~\ref{sec:degrees}: the first records \emph{whether} a world is grounded, the second \emph{how dispersed} its grounding is.

\begin{defi}\label{def:situatedness}
Fix a model and a world $d$. The \emph{degree of groundedness} at $d$ is $g(d)=|\mathcal{P}_d|/|\mathcal{P}|$, the fraction of atoms with a proper referent. Calling an unordered pair $\{p,q\}\subseteq\mathcal{P}_d$ \emph{separated at $d$} when $\sigma_d(p)$ and $\sigma_d(q)$ are $\le$-incomparable, the \emph{degree of situatedness} at $d$ is
\[
\mathrm{sit}(d)=\frac{\big|\{\{p,q\}\subseteq\mathcal{P}_d:\ \sigma_d(p)\not\le\sigma_d(q)\ \text{and}\ \sigma_d(q)\not\le\sigma_d(p)\}\big|}{\binom{|\mathcal{P}_d|}{2}},
\]
with $\mathrm{sit}(d)=0$ when fewer than two atoms have proper referents.
\end{defi}

By Proposition~\ref{prop:collapse} the collapse of $\Vdash_\sigma$ onto $\Vdash$ at $d$ is governed by $g$ alone, occurring exactly when $g(d)=0$. The index $\mathrm{sit}(d)$ measures a finer quantity, the dispersion of the proper referents, and its vanishing is necessary but not sufficient for the collapse: $g(d)=0$ forces $\mathrm{sit}(d)=0$, but not conversely. The two indices part company on \emph{nested} anchoring. A text whose markers point to a $\le$-chain of proper situations --- a room within a house within a city --- has $g(d)>0$ and yet $\mathrm{sit}(d)=0$, since no two of its referents are incomparable; there $\Vdash_\sigma$ is still strictly finer than $\Vdash$, a premise about the room failing to reach the situation of the city that contains it. A placeless world has $g(d)=\mathrm{sit}(d)=0$; a world whose markers point to pairwise-incomparable proper situations has $g(d)>0$ and $\mathrm{sit}(d)=1$. The geographical claim thus resolves into two measurable components, groundedness and dispersion, both of which stand on the same footing as the relational closure to which we now turn.

\section{The empirical setting}\label{sec:empirical}

We fix the bridge between this semantics and measurement. The original
corpus comprised $64$ human-authored Italian texts and $192$ texts generated
by a large language model (GPT-3.5) from thematic cues extracted from the
human responses, at a planned ratio of one human text to three generated
texts. Nine generated texts, corresponding to three prompts, were unavailable.
To preserve the prompt structure of the study, the empirical analysis was
therefore restricted to the $61$ complete prompt clusters. The resulting
analytical sample contains $244$ texts: $61$ human texts and $183$
AI-generated texts, with three generated responses associated with each
retained human prompt.

Each text $w_i$ is represented by an embedding
$\mathbf{x}_i\in\mathbb{R}^{768}$ obtained using the multilingual sentence
transformer
\texttt{sentence-transformers/paraphrase-
multilingual-mpnet-base-v2}.
For a fixed neighbourhood size $k$, cosine distance induces the accessibility
relation
\[
w_iR_kw_j
\quad\text{iff}\quad
w_j \text{ is among the $k$ nearest neighbours of } w_i,
\]
where the trivial self-neighbour is excluded. Unless otherwise stated, the
baseline specification uses $k=10$.

The pair $\mathcal{F}_k=(W,R_k)$ is a finite frame realising, concretely,
the accessibility reduct of a situation frame: the abstract portion of the
world accessible from $d$ becomes the set of texts semantically nearest to a
given one. A valuation separating human from artificial texts turns
$\mathcal{F}_k$ into a model. We first report the relational properties of
the two induced subframes formed by the $61$ human and $183$ artificial
texts. Because these subframes have different numbers of vertices, the main
comparative analysis subsequently uses a prompt-balanced resampling design
with equal-sized human and artificial graphs.

The local organisation of $R_k$ is summarised by the frequencies with which it satisfies the standard relational properties. Table~\ref{tab:frame} reports them for the two subcorpora; these are the figures on which the entire argument turns.

\begin{table}[htbp]
\centering
\caption{Relational frequencies in the complete-prompt human and artificial
subframes, using $k=10$ and cosine distance.}
\label{tab:frame}
\begin{tabular}{lccc}
\toprule
Property & Human $(n=61)$ & AI-generated $(n=183)$ & Difference \\
\midrule
Symmetry       & 0.603 & 0.613 & 0.010 \\
Transitivity   & 0.372 & 0.487 & 0.115 \\
Euclideanity   & 0.413 & 0.541 & 0.128 \\
Seriality      & 1.000 & 1.000 & 0.000 \\
\bottomrule
\end{tabular}
\end{table}

At the descriptive level, the artificial subframe exhibits higher
transitivity and Euclideanity than the human subframe. The corresponding
differences are $0.115$ and $0.128$, respectively. The difference in symmetry
is much smaller, equal to $0.010$, while seriality is complete in both
subframes because every vertex has $k$ outgoing neighbours by construction.
These comparisons must nevertheless be interpreted cautiously because the
human and artificial graphs contain different numbers of vertices. The
dependence-aware and graph-size-balanced comparison is therefore reported
below. The relational frequencies sit on top of a layer of global geometric statistics, recorded in Table~\ref{tab:global}, which describe the embedding configuration before any accessibility relation is imposed: cluster separation, the fidelity with which the nearest-neighbour structure of the ambient space is preserved on the subcorpus, the mean within-class cosine similarity, and the variance captured by the leading principal components. These are the output of the standard apparatus for reading high-dimensional semantic geometry --- principal-component and spectral clustering and the manifold-learning family of nonlinear projections --- and they belong to the metric, not the deductive, life of the corpus. Only the relational frequencies of Table~\ref{tab:frame} and the boundary structure of Table~\ref{tab:boundary} are absorbed by the logic.

\begin{table}[htbp]
\centering
\begin{tabular}{|l|c|c|}
\hline
Descriptor & Human & AI-generated \\
\hline
Silhouette (cluster coherence) & 0.62 & 0.41 \\
Neighbourhood preservation & 0.85 & 0.72 \\
Intra-cluster cosine similarity & 0.78 & 0.61 \\
Variance in leading PCA components & 0.67 & 0.48 \\
\hline
\end{tabular}
\caption{Global geometric descriptors of the embedding configuration.}
\label{tab:global}
\end{table}

The interior and boundary structure of the two subframes --- the proportion of texts all of whose neighbours share their class, against those whose neighbourhood is mixed --- is given in Table~\ref{tab:boundary} together with the local class heterogeneity, whose modal reading is established in Section~\ref{sec:semantics}.

\begin{table}[htbp]
\centering
\begin{tabular}{|l|c|c|}
\hline
Property & Human & AI-generated \\
\hline
Boundary rate & 0.922 & 0.623 \\
Interior rate & 0.016 & 0.377 \\
Local modal entropy & 0.783 & 0.377 \\
\hline
\end{tabular}
\caption{Interior, boundary and local heterogeneity of the subframes.}
\label{tab:boundary}
\end{table}

The global accessibility profile --- closeness centrality $0.235$ against $0.258$, betweenness centrality $0.0118$ against $0.0062$ --- belongs to the graph-theoretic rather than the deductive register, and we retain it only as a reminder that the logic recovers one face of the data and not all of it.

\section{Modal semantics of the frames}\label{sec:semantics}

To speak of validation we must first have something to validate. The language is extended by the necessity operator \cite{ref_poggiolesi2011}, with $\Box\phi$ true at a world when $\phi$ holds throughout the accessible neighbourhood and $\Diamond\phi:=\neg\Box\neg\phi$ its dual. The topological vocabulary of the empirical analysis is then merely the standard translation at work \cite{ref_blackburn2001}: with $\alpha$ the atom whose extension is a class $A$ of texts, the interior of $A$ is the extension of $\Box\alpha$ and its boundary the extension of $\Diamond\alpha\wedge\Diamond\neg\alpha$, so that the rates of Table~\ref{tab:boundary} are frequencies of definable modal conditions.

That observation can be sharpened until the third row of Table~\ref{tab:boundary}, the local modal entropy, acquires the same status. For a world $w$ with neighbourhood $R_k(w)$ of size $k$, write $p_w=|R_k(w)\cap A|/k$ for the fraction of its neighbours lying in the class $A$, and let $H(p)=-p\log_2 p-(1-p)\log_2(1-p)$ be the binary Shannon entropy of the class label across the neighbourhood \cite{ref_covertthomas2006}.

\begin{prop}\label{prop:entropy}
At every world $w$,
\[
w\models\Box\alpha\vee\Box\neg\alpha \iff p_w\in\{0,1\} \iff H(p_w)=0,
\]
\[
w\models\Diamond\alpha\wedge\Diamond\neg\alpha \iff 0<p_w<1 \iff H(p_w)>0.
\]
Consequently the interior rate of Table~\ref{tab:boundary} is the frequency of $\{H=0\}$, the boundary rate the frequency of $\{H>0\}$, and the local modal entropy is the average over $W$ of a functional supported exactly on the modal boundary.
\end{prop}

\textbf{Proof.} $\Box\alpha$ holds at $w$ iff every $R_k$-neighbour of $w$ lies in $A$, i.e. $p_w=1$; dually $\Box\neg\alpha$ iff $p_w=0$, so $\Box\alpha\vee\Box\neg\alpha$ iff $p_w\in\{0,1\}$. Symmetrically $\Diamond\alpha$ iff $p_w>0$ and $\Diamond\neg\alpha$ iff $p_w<1$, so the conjunction holds iff $0<p_w<1$. Since $H$ vanishes precisely at $0$ and $1$ and is positive between them, the entropy equivalences follow, and the frequency statements are the corresponding counts over $W$. $\blacksquare$

Proposition~\ref{prop:entropy} makes the information-theoretic descriptor a modal one: the information carried by a world's neighbourhood is literally the modal indeterminacy of its class, vanishing on the interior and maximal where the two classes meet in equal measure. The human frame, with a boundary rate of $0.922$ and a mean entropy of $0.783$, is almost everywhere indeterminate in this sense; the artificial frame, interior on more than a third of its worlds and with mean entropy $0.377$, is correspondingly settled. This is the same closure that Table~\ref{tab:frame} measures, seen now through the class label rather than the relational atoms.

It is worth pausing on what the modal layer adds to the situation-based system, because at first it adds nothing. The truth clauses of Definition~\ref{def:reference} are classical and pointwise; neither the accessibility $R$, nor the assignment $S$, nor the reference function $\sigma$ enters them.

\begin{prop}\label{prop:inert}
On situation models, $\Delta\Vdash\phi$ holds if and only if $\Delta$ classically entails $\phi$. In particular ordinary consequence is blind to $R$, $S$ and $\sigma$.
\end{prop}

\textbf{Proof.} The truth value of a formula at a world depends only on the Boolean values of its atoms there, so it is computed by the propositional valuation $d$ induces. Classical entailment therefore suffices for $\Vdash$, and conversely a propositional countermodel is realised at a single world carrying the offending valuation. $\blacksquare$

Proposition~\ref{prop:inert} is not an objection but a diagnosis: it isolates exactly what one must switch on for the situational apparatus to bear on inference. The reference function is switched on by the situated relation $\Vdash_\sigma$ of Definition~\ref{def:situated}, and graded by $\mathrm{sit}$ of Definition~\ref{def:situatedness}; the accessibility relation is switched on by the modal operator, and here the gain is genuine, for $\Box$ quantifies over $R_k$ and lets the geometry of the neighbourhood graph speak. What it says is governed by the classical correspondence between modal axioms and frame conditions.

\begin{thm}\label{thm:corr}
Over the class of frames, each of the following axioms is valid exactly on the frames whose accessibility relation has the stated property: $\mathsf{T}\;(\Box A\to A)$ on the reflexive frames; $\mathsf{B}\;(A\to\Box\Diamond A)$ on the symmetric; $\mathsf{4}\;(\Box A\to\Box\Box A)$ on the transitive; $\mathsf{5}\;(\Diamond A\to\Box\Diamond A)$ on the Euclidean; and $\mathsf{D}\;(\Box A\to\Diamond A)$ on the serial.
\end{thm}

\textbf{Proof.} These are the Sahlqvist correspondences \cite{ref_blackburn2001}. Validity follows from the truth clauses; for the converse one refutes the axiom on any frame violating the condition by placing a suitable valuation at the offending tuple. $\blacksquare$

A $k$-nearest-neighbour relation has none of these properties by construction. It need not be symmetric, since $w_j$ may be among the $k$ nearest to $w_i$ without the reverse holding; it is almost never transitive or Euclidean at the scale of a whole corpus; and, excluding as it does the trivial self-edge, it is not even reflexive. The class of semantic frames is thus the class of arbitrary finite frames, whose modal logic is the basic system $\mathsf{K}$. Theorem~\ref{thm:corr} therefore does not apply to $\mathcal{F}_k$ as an all-or-nothing verdict; it applies, as we now argue, by degrees.

\section{Degrees of validation}\label{sec:degrees}

The frame conditions of Theorem~\ref{thm:corr} are first-order sentences of a uniform shape: a universally quantified implication between conjunctions of relational atoms, degenerating to a reflexive or a serial clause for $\mathsf{T}$ and $\mathsf{D}$. A finite frame either satisfies such a sentence or does not, but it also satisfies it to a measurable extent, namely on the proportion of the relevant tuples for which the implication does not fail. This proportion is the natural carrier of the empirical content of Table~\ref{tab:frame}.

\begin{defi}\label{def:degree}
Let $F=(W,R)$ be a finite frame and let $X$ be one of $\mathsf{T},\mathsf{B},\mathsf{4},\mathsf{5},\mathsf{D}$, with frame condition of antecedent $\varphi_X$ and consequent $\psi_X$ over the displayed tuple of worlds. The \emph{degree of validation} of $X$ on $F$ is
\[
\deg_X(F)=\frac{\big|\{\bar a : F\models\varphi_X(\bar a)\ \text{and}\ F\models\psi_X(\bar a)\}\big|}{\big|\{\bar a : F\models\varphi_X(\bar a)\}\big|},
\]
with the convention $\deg_X(F)=1$ when the antecedent is nowhere satisfied. Explicitly, $\deg_{\mathsf{B}}$ counts the pairs with $wRv$ for which also $vRw$; $\deg_{\mathsf{4}}$ the composable pairs $wRv,vRu$ for which also $wRu$; and $\deg_{\mathsf{T}},\deg_{\mathsf{D}}$ the worlds that are, respectively, reflexive and serial, over all of $W$.
\end{defi}

The Euclidean condition $\forall w\,v\,u\,(wRv\wedge wRu\to vRu)$ requires a separate word, because its quantified triple $(w,v,u)$ ranges over co-initial pairs of two kinds. The \emph{proper} pairs, with $v\neq u$, ask that two distinct worlds accessible from a common source be mutually accessible; the \emph{diagonal} pairs, with $v=u$, collapse the condition to $wRv\to vRv$ and so demand a self-loop on the target of every edge. The two are counted by two degrees:
\[
\deg_{\mathsf{5}}(F)=\frac{\big|\{(w,v,u):wRv,\ wRu,\ v\neq u,\ vRu\}\big|}{\big|\{(w,v,u):wRv,\ wRu,\ v\neq u\}\big|},
\qquad
\deg_{\mathsf{5}^{c}}(F)=\frac{\big|\{(w,v):wRv,\ vRv\}\big|}{\big|\{(w,v):wRv\}\big|},
\]
the proper, off-diagonal degree $\deg_{\mathsf{5}}$ measuring the genuine relational closure of distinct co-initial neighbours, and the diagonal degree $\deg_{\mathsf{5}^{c}}$ coinciding with the frequency of self-loops on edge-targets. It is $\deg_{\mathsf{5}}$ that the Euclideanity column of Table~\ref{tab:frame} records; $\deg_{\mathsf{5}^{c}}$, whose value in a self-loop-free $k$-nearest-neighbour frame is identically zero, is displayed nowhere in the tables and enters the account only through the calculus of Section~\ref{sec:calculus}. The reported statistic is thus an off-diagonal closure measure, not a complete degree of validation of $\mathsf{5}$; the two are related in Proposition~\ref{prop:extremes}.

\begin{prop}\label{prop:tableisdegree}
The Symmetry, Transitivity, Euclideanity and Seriality columns of Table~\ref{tab:frame} are, respectively, $\deg_{\mathsf{B}}$, $\deg_{\mathsf{4}}$, $\deg_{\mathsf{5}}$ and $\deg_{\mathsf{D}}$ of the human and artificial subframes.
\end{prop}

\textbf{Proof.} Each reported frequency is, by the construction of Table~\ref{tab:frame}, the ratio of the tuples witnessing the property to the tuples eligible for it, which is the ratio of Definition~\ref{def:degree}. $\blacksquare$

\begin{prop}\label{prop:extremes}
For $X\in\{\mathsf{T},\mathsf{B},\mathsf{4},\mathsf{D}\}$ and every finite frame $F$, one has $\deg_X(F)=1$ if and only if $F$ satisfies the frame condition of $X$, hence if and only if $X$ is valid on $F$; the complement $1-\deg_X(F)$ measures the density of tuples on which the condition fails, and $\deg_X(F)=0$ when the property holds nowhere it is eligible. For Euclideanity the same role is played by the pair $(\deg_{\mathsf{5}},\deg_{\mathsf{5}^{c}})$: axiom $\mathsf{5}$ is valid on $F$ if and only if $\deg_{\mathsf{5}}(F)=\deg_{\mathsf{5}^{c}}(F)=1$, so that $\deg_{\mathsf{5}}(F)=1$ alone certifies the off-diagonal closure only, and not the full frame condition.
\end{prop}

\textbf{Proof.} For $X\in\{\mathsf{T},\mathsf{B},\mathsf{4},\mathsf{D}\}$ the numerator of Definition~\ref{def:degree} equals its denominator precisely when every eligible tuple satisfies the consequent, which is the frame condition; the equivalence with validity is Theorem~\ref{thm:corr}, and the remaining clauses are the complementary count. For Euclideanity the frame condition $\forall w\,v\,u\,(wRv\wedge wRu\to vRu)$ splits, according to whether $v=u$, into the proper closure counted by $\deg_{\mathsf{5}}$ and the diagonal clause $wRv\to vRv$ counted by $\deg_{\mathsf{5}^{c}}$; every quantified triple falls under one of the two, so the condition holds throughout exactly when both numerators meet their denominators, that is, when $\deg_{\mathsf{5}}(F)=\deg_{\mathsf{5}^{c}}(F)=1$, which by Theorem~\ref{thm:corr} is the validity of $\mathsf{5}$. $\blacksquare$

A degree of validation is therefore an interpolation between non-validity and validity.

Writing the four degrees of
Proposition~\ref{prop:tableisdegree} as the profile vector
$(\deg_{\mathsf{B}},\deg_{\mathsf{4}},
\deg_{\mathsf{5}},\deg_{\mathsf{D}})$, in which $\deg_{\mathsf{5}}$ is the
proper Euclidean degree of Definition~\ref{def:degree}, the complete-prompt
subframes yield
\[
\text{human: } (0.603,\,0.372,\,0.413,\,1.000),
\qquad
\text{artificial: } (0.613,\,0.487,\,0.541,\,1.000).
\]
The largest descriptive differences concern transitivity and Euclideanity.
Symmetry differs only weakly in the unequal-sized full subframes, whereas
seriality is complete in both groups by construction.

\begin{table}[htbp]
\centering
\caption{Descriptive modal differences in the complete-prompt subframes
using $k=10$.}
\label{tab:magnitude}
\begin{tabular}{lccc}
\toprule
Axiom & Human & AI-generated & AI--Human difference \\
\midrule
$\mathsf{B}$ & 0.603 & 0.613 & 0.010 \\
$\mathsf{4}$ & 0.372 & 0.487 & 0.115 \\
$\mathsf{5}$ & 0.413 & 0.541 & 0.128 \\
$\mathsf{D}$ & 1.000 & 1.000 & 0.000 \\
\bottomrule
\end{tabular}
\end{table}

These values are descriptive properties of the observed unequal-sized
subframes. They are not treated as independent proportions and are not used
to construct binomial standard errors or conventional two-sample tests.
The inferential comparison instead balances the number of vertices across
groups and preserves the common prompt structure.

\subsection{The degree as an estimator}\label{subsec:estimator}

The inferential target considered throughout this paper is not the hypothetical population of all possible human or AI-generated texts, but the distribution of modal validation degrees induced by the prompt-generation mechanism underlying the observed corpus. The reported degrees should therefore be interpreted as empirical plug-in estimates of the corresponding population validation probabilities under this data-generating process.
So far $\deg_X(F)$ is a deterministic ratio evaluated on an observed
finite frame. If the frame is regarded as one realisation of a broader
semantic-generative process, the corresponding population quantity may be
written as
\[
\pi_X
=
\Pr\!\left(
F\models\psi_X(\bar a)
\mid
F\models\varphi_X(\bar a)
\right),
\]
and the observed degree
\[
\widehat{\pi}_X=\deg_X(F)
\]
is its empirical plug-in estimate.

This representation does not imply that eligible relational tuples constitute
independent Bernoulli observations. In a nearest-neighbour graph, different
tuples may share vertices, directed edges, paths, and local neighbourhoods.
The dependence is particularly strong for transitivity and Euclideanity,
whose eligible configurations are constructed from overlapping pairs of
edges. Moreover, the corpus has a clustered prompt structure: each retained
prompt contributes one human text and three generated texts. Ordinary
binomial standard errors and tuple-level two-sample tests would therefore
understate the dependence present in the data.

\subsubsection{Prompt-balanced resampling}

The principal comparative analysis addresses both graph-size imbalance and
the prompt structure. All $61$ human texts are retained in every replication.
For each of the $61$ complete prompts, one of the three associated
AI-generated texts is selected at random. Each replication therefore compares
a human graph and an artificial graph containing the same number of vertices,
the same prompts, the same neighbourhood parameter $k$, and the same graph
density $k/(n-1)$.

The procedure was repeated $5{,}000$ times with random seed $12345$. For each
replication, the artificial $k$-nearest-neighbour graph was reconstructed and
its symmetry, transitivity, and Euclideanity degrees were recomputed. The
human graph remained fixed because each prompt has one observed human text.
For each modal property, uncertainty is summarised using the empirical
standard deviation and the $2.5$th and $97.5$th percentiles of the resampling
distribution. We also report the proportion of replications in which the
artificial degree exceeds the corresponding human degree. These percentile
ranges describe variation induced by the selection of one artificial response
per prompt; they are not binomial confidence intervals based on independent
relational tuples.

\begin{table}[htbp]
\centering
\caption{Prompt-balanced comparison of human and AI-generated frames.
Each replication contains $61$ human and $61$ artificial texts; $5{,}000$
replications, cosine distance and $k=10$.}
\label{tab:balanced}
\begin{tabular}{lccccc}
\toprule
Property
& Human
& Mean AI
& Mean difference
& 95\% percentile range
& $\Pr(\mathrm{AI}>\mathrm{H})$ \\
\midrule
Symmetry
& 0.603
& 0.643
& 0.040
& [0.000,\ 0.079]
& 0.970 \\

Transitivity
& 0.372
& 0.447
& 0.076
& [0.049,\ 0.104]
& 1.000 \\

Euclideanity
& 0.413
& 0.497
& 0.084
& [0.054,\ 0.116]
& 1.000 \\
\bottomrule
\end{tabular}
\end{table}

After balancing graph size, the artificial frames continue to exhibit greater
transitivity and Euclideanity in every resampling replication. The mean
differences are $0.076$ for transitivity and $0.084$ for Euclideanity, and
their percentile ranges remain entirely above zero. The symmetry difference
is smaller and less stable: its mean is $0.040$, its lower percentile bound is
zero, and the artificial value exceeds the human value in $97.0\%$ of the
replications. The principal empirical evidence therefore concerns
transitivity and Euclideanity rather than symmetry.

\subsubsection{Robustness analyses}

Three complementary analyses were conducted. First, a leave-one-prompt-out
jackknife removed one complete prompt cluster at a time, where a cluster
contains one human and three artificial texts. Across the $61$ jackknife
replications, the estimated AI--human differences remained positive on
average:
\[
\begin{aligned}
\mathsf{B}:&\quad 0.008
\quad (\mathrm{SE}_{\mathrm{jack}}=0.055),\\
\mathsf{4}:&\quad 0.113
\quad (\mathrm{SE}_{\mathrm{jack}}=0.049),\\
\mathsf{5}:&\quad 0.126
\quad (\mathrm{SE}_{\mathrm{jack}}=0.054).
\end{aligned}
\]
The jackknife is used as an influence analysis showing that the observed
differences are not generated by a single thematic prompt.

Second, the prompt-balanced resampling analysis was repeated for
$k\in\{5,8,10,12,15\}$. The mean AI--human transitivity differences were,
respectively,
\[
0.111,\quad 0.074,\quad 0.076,\quad 0.071,\quad 0.063,
\]
and the corresponding Euclideanity differences were
\[
0.139,\quad 0.084,\quad 0.084,\quad 0.077,\quad 0.068.
\]
For both properties the artificial degree exceeded the human degree in all
$5{,}000$ replications at every examined value of $k$. Symmetry was less
robust, with a mean difference of $-0.005$ at $k=5$ and positive mean
differences between $0.034$ and $0.050$ for $k=8,\ldots,15$.

Third, cosine distance was replaced by Euclidean distance after
$\ell_2$-normalising every embedding. The resulting neighbourhood graphs and
modal differences were identical: the replication-level correlations between
the cosine- and Euclidean-based differences were $1.000$ for symmetry,
transitivity, and Euclideanity. This equivalence is expected because, for
unit-normalised vectors, cosine distance and squared Euclidean distance are
monotone transformations of one another.

Taken together, these analyses show that the higher artificial transitivity
and Euclideanity are not explained by the original $1{:}3$ graph-size
imbalance, by one influential prompt, by the selected neighbourhood size, or
by the distinction between cosine and Euclidean distance on normalised
embeddings.

\begin{rem}\label{rem:graded}
The assignment $X\mapsto\deg_X(F)$ is a $[0,1]$-valued refinement of the Sahlqvist correspondence of Theorem~\ref{thm:corr}: the bivalent verdict ``$X$ is valid on $F$'' is recovered as the event $\deg_X(F)=1$, and intermediate values quantify the measure-theoretic extent of validity. This is not a many-valued logic in the algebraic sense --- the degrees are not truth-values closed under the connectives --- but a \emph{statistical valuation of axioms}, a frequency of correspondence, whose endpoints are the two faces of Theorem~\ref{thm:corr}.
\end{rem}

What these numbers do not yet say is what such an interpolation \emph{means} for inference, and that requires the calculus.

\section{From frequencies to syntactic constraints: the calculus $\mathsf{G3.K}$}\label{sec:calculus}

The labelled sequent calculus internalises the relational semantics by admitting two kinds of object into a sequent: \emph{labelled formulas} $w{:}A$, asserting that $A$ holds at the world named $w$, and \emph{relational atoms} $wRv$, asserting accessibility. A sequent $\Gamma\Rightarrow\Delta$ has finite multisets of such objects on each side, and the base calculus $\mathsf{G3.K}$ consists of the initial sequents $w{:}p,\Gamma\Rightarrow\Delta,w{:}p$ and $w{:}\bot,\Gamma\Rightarrow\Delta$, the usual left and right rules for the connectives applied at a label, and the four modal rules
\[
\frac{u{:}A,\ w{:}\Box A,\ wRu,\ \Gamma\Rightarrow\Delta}{w{:}\Box A,\ wRu,\ \Gamma\Rightarrow\Delta}\,L\Box
\qquad
\frac{wRv,\ \Gamma\Rightarrow\Delta,\ v{:}A}{\Gamma\Rightarrow\Delta,\ w{:}\Box A}\,R\Box
\]
\[
\frac{wRu,\ u{:}A,\ \Gamma\Rightarrow\Delta}{w{:}\Diamond A,\ \Gamma\Rightarrow\Delta}\,L\Diamond
\qquad
\frac{wRv,\ \Gamma\Rightarrow\Delta,\ w{:}\Diamond A,\ v{:}A}{wRv,\ \Gamma\Rightarrow\Delta,\ w{:}\Diamond A}\,R\Diamond
\]
with $R\Box$ and $L\Diamond$ carrying the eigenvariable condition that the accessible label ($v$ in $R\Box$, $u$ in $L\Diamond$) not occur in the conclusion. Read from conclusion to premise, $R\Box$ reduces a necessity at $w$ to a claim about an arbitrary accessible world and $L\Box$ discharges a necessity along an edge already present; the $\Diamond$ rules are their duals, taking $\Diamond$ as primitive rather than defined. The principal formula is repeated in the premise of $L\Box$ and $R\Diamond$ so that the boxed or diamonded formula remains available for further use. This is the calculus of Negri \cite{ref_negri2005}, governed throughout by a single discipline: the rules repeat their principal formulas in the premises, so that the structural rules need not be primitive. In Poggiolesi's terms these are \emph{external}, or ``internalised forcing'', calculi \cite{ref_poggiolesi2011}: one never derives a bare modal formula $\Box A$ but only labelled objects $w{:}\Box A$ together with relational atoms, so the subformula property claimed below is that of the extended labelled language rather than of the pure modal one --- a feature that is, for us, not a cost but the very hinge of the argument, since it is precisely the relational atoms that the licensing reading of Section~\ref{sec:degrees} counts.

\begin{thm}\label{thm:structural}
In $\mathsf{G3.K}$ all rules are height-preserving invertible, weakening and contraction are height-preserving admissible, and cut is admissible; the calculus is therefore cut-free and enjoys the subformula property in the labelled language.
\end{thm}

\textbf{Proof.} The argument is the one given for the $\mathsf{G3}$ family in Negri \cite{ref_negri2005} and Negri and von Plato \cite{ref_negrivonplato2011}. Invertibility and the admissibility of weakening are inductions on derivation height; contraction is height-preserving admissible by induction on height using invertibility, the repetition of principal formulas absorbing the duplicated occurrences; and cut is eliminated by the standard induction on the pair consisting of the weight of the cut formula and the cut height, the only modal principal case being a cut on $w{:}\Box A$ between $R\Box$ and $L\Box$, reduced to a cut on $v{:}A$ of smaller weight after the eigenlabel of $R\Box$ is instantiated to the side label of $L\Box$. $\blacksquare$

These properties survive the addition of frame conditions, provided the conditions are written as rules of the right shape. A \emph{geometric} frame condition --- a sentence $\forall\bar z\,(\varphi\supset\psi)$ whose $\varphi,\psi$ are built from atoms by $\wedge,\vee,\exists$ alone, the form of all five of our conditions --- is internalised by a rule that, reading upward, \emph{adds} the atoms witnessing its consequent to the antecedent while retaining the atoms of its antecedent; where the consequent is existential, as for seriality, the witness is a fresh eigenvariable:
\[
\frac{wRw,\ \Gamma\Rightarrow\Delta}{\Gamma\Rightarrow\Delta}\,\mathrm{Ref}
\qquad
\frac{vRw,\ wRv,\ \Gamma\Rightarrow\Delta}{wRv,\ \Gamma\Rightarrow\Delta}\,\mathrm{Sym}
\qquad
\frac{wRv,\ \Gamma\Rightarrow\Delta}{\Gamma\Rightarrow\Delta}\,\mathrm{Ser}
\]
\[
\frac{wRu,\ wRv,\ vRu,\ \Gamma\Rightarrow\Delta}{wRv,\ vRu,\ \Gamma\Rightarrow\Delta}\,\mathrm{Trans}
\qquad
\frac{vRu,\ wRv,\ wRu,\ \Gamma\Rightarrow\Delta}{wRv,\ wRu,\ \Gamma\Rightarrow\Delta}\,\mathrm{Eucl}
\]
\[
\frac{vRv,\ wRv,\ \Gamma\Rightarrow\Delta}{wRv,\ \Gamma\Rightarrow\Delta}\,\mathrm{Eucl}^{c}
\]
with $\mathrm{Ser}$ subject to the eigenvariable condition on $v$, and $\mathrm{Eucl}^{c}$ the \emph{contracted instance} of $\mathrm{Eucl}$ obtained by identifying $u$ with $v$. Negri's theorem on geometric extensions guarantees that $\mathsf{G3.K}$ enlarged by any selection of these rules retains every clause of Theorem~\ref{thm:structural}, provided each rule is adjoined together with its contracted instances: the new rules act only on relational atoms, never on the cut formula, and the \emph{closure condition} --- that any contraction of two principal relational atoms in a conclusion again be an instance of a rule of the system --- is what keeps contraction, and with it cut, admissible \cite{ref_negri2003,ref_negrivonplato2011}. Among our five conditions this bites only for Euclideanity: its contracted instance $\mathrm{Eucl}^{c}$, the diagonal case $v=u$ that turns an edge $wRv$ into the self-loop $vRv$, must be present, while the reflexive, symmetric, transitive and serial rules are closed already. The correspondence is the expected one --- $\mathrm{Ref}$ realises $\mathsf{T}$, $\mathrm{Sym}$ realises $\mathsf{B}$, $\mathrm{Trans}$ realises $\mathsf{4}$, the pair $\{\mathrm{Eucl},\mathrm{Eucl}^{c}\}$ realises $\mathsf{5}$, and $\mathrm{Ser}$ realises $\mathsf{D}$ --- so that the columns of Table~\ref{tab:frame} stand in bijection with the non-logical rules one may adjoin.

Here the two halves of the paper meet. A relational rule is a schema, and over a fixed finite frame it has ground instances, obtained by interpreting its schematic labels as worlds so that the antecedent atoms hold in the frame. The instance asserts the consequent atom; whether the frame honours that assertion is exactly the question Definition~\ref{def:degree} counts.

\begin{prop}\label{prop:licensing}
Let $X$ be one of $\mathsf{T},\mathsf{B},\mathsf{4},\mathsf{5},\mathsf{D}$ and $\rho_X$ its relational rule. For a finite frame $F$, the degree $\deg_X(F)$ is the proportion of the ground instances of $\rho_X$ over $F$ --- those whose antecedent atoms hold in $F$ --- whose consequent atom also holds in $F$. Consequently $\rho_X$ is sound over $F$ if and only if $\deg_X(F)=1$; for $\deg_X(F)<1$ the instances counted by $\deg_X(F)$ are exactly those whose addition to $\mathsf{G3.K}$ is sound over $F$, and the remaining fraction $1-\deg_X(F)$ measures the syntactic overreach incurred by imposing $X$ globally on $F$. Writing $\mathcal{L}_X(F)$ for the set of sound ground instances, the calculus $\mathsf{G3.K}+\bigcup_X\mathcal{L}_X(F)$ is sound over $F$ by construction, retains cut-elimination --- each licensed instance is a geometric rule of the admissible shape of Theorem~\ref{thm:structural} --- and contains $\sum_X \deg_X(F)\cdot N_X$ licensed relational instances in all.
\end{prop}

\textbf{Proof.} An instance of $\rho_X$ over $F$ is determined by an interpretation of its schematic labels satisfying the antecedent atoms, that is, by a tuple $\bar a$ with $F\models\varphi_X(\bar a)$; the instance is sound over $F$ precisely when its consequent atom holds, that is, when $F\models\psi_X(\bar a)$. The proportion of sound instances among applicable ones is therefore the ratio of Definition~\ref{def:degree}, namely $\deg_X(F)$. Global soundness of $\rho_X$ is soundness of every instance, equivalent to $\deg_X(F)=1$; the complementary clause is the count of unsound instances. Each licensed instance acts only on relational atoms and retains its antecedent atoms, so the extension is geometric and Theorem~\ref{thm:structural} applies; the total multiplies the licensed fraction $\deg_X(F)$ by the number $N_X$ of applicable instances and sums over $X$. $\blacksquare$

This is the transformation we sought, and the cut-elimination clause is what makes it well defined: because the licensed fragment is a geometric extension, its licensing density is an invariant of the calculus and not an accident of a particular proof search. The transitivity entries of Table~\ref{tab:frame} are no longer merely geometric descriptors: they quantify the fraction of applications of the rule $\mathrm{Trans}$ that each observed subframe licenses. The descriptive difference between the human and artificial corpora is therefore interpreted proof-theoretically as a difference in rule licensing, while its empirical uncertainty is evaluated separately through the dependence-aware prompt-balanced resampling procedure described in Section~\ref{subsec:estimator}. For Euclideanity the relational rule is the pair $\{\mathrm{Eucl},\mathrm{Eucl}^{c}\}$, to each member of which Proposition~\ref{prop:licensing} applies in turn: the proper rule $\mathrm{Eucl}$ is licensed to the degree $\deg_{\mathsf{5}}$ recorded in Table~\ref{tab:frame}, its contracted instance $\mathrm{Eucl}^{c}$ to the degree $\deg_{\mathsf{5}^{c}}$, so that the full axiom is licensed only where both densities reach one --- which, since $\deg_{\mathsf{5}^{c}}$ vanishes on a self-loop-free frame, they never do. A frequency has thereby become a constraint on a derivation without relying on independence assumptions for the underlying relational tuples.

The artificial corpus, validating seriality outright and exhibiting comparatively higher transitivity and Euclideanity than the human corpus, supports a calculus more densely populated by the rules associated with transitive--Euclidean systems, so that its inferential regime presses towards $\mathsf{K45}$ and, with seriality, $\mathsf{KD45}$ --- the logics of a closed, self-reinforcing space of mutually accessible worlds. The human corpus licenses these same rules so sparingly that little is added to the bare core, and its regime remains close to $\mathsf{K}$, where necessity reaches only as far as the explicitly present edges. One should resist the temptation to push the artificial corpus all the way to $\mathsf{S5}$: that system requires reflexivity, and the $k$-nearest-neighbour construction suppresses the self-edge by fiat, so the column that would carry $\deg_{\mathsf{T}}$ is structurally near zero and the last step from $\mathsf{KD45}$ to $\mathsf{S5}$ is one the observed graph does not support. The reported Euclidean density $\deg_{\mathsf{5}}$ of Table~\ref{tab:frame} is carried entirely by the proper, off-diagonal instances of $\mathrm{Eucl}$ ($v\neq u$); its contracted companion $\mathrm{Eucl}^{c}$, whose licensing density is the diagonal degree $\deg_{\mathsf{5}^{c}}$ of Definition~\ref{def:degree}, demands a self-loop $vRv$ on the target of every edge and so shares the fate of reflexivity, its density near zero in a self-loop-free frame. The suppression of $\mathsf{S5}$ is therefore twice determined --- once through $\deg_{\mathsf{T}}$, once through the contracted Euclidean fragment --- which is a defect not of the corpus but of the instrument, and the kind of thing the present reading makes it possible to say precisely.

It remains to see that an actual proof carries this graded weight. Consider factivity, the axiom $\mathsf{T}$, whose derivation exhibits the method in miniature.
\begin{prooftree}
\AxiomC{$w{:}A,\ wRw,\ w{:}\Box A \Rightarrow w{:}A$}
\RightLabel{$L\Box$}
\UnaryInfC{$wRw,\ w{:}\Box A \Rightarrow w{:}A$}
\RightLabel{$\mathrm{Ref}$}
\UnaryInfC{$w{:}\Box A \Rightarrow w{:}A$}
\RightLabel{$R{\to}$}
\UnaryInfC{$\Rightarrow w{:}\Box A \to A$}
\end{prooftree}
The derivation is correct in $\mathsf{G3.K}+\mathrm{Ref}$, and its single use of a non-logical rule is the step labelled $\mathrm{Ref}$, which posits the reflexive edge $wRw$. By Proposition~\ref{prop:licensing} that step is licensed by a frame to the degree $\deg_{\mathsf{T}}$, and we have just seen that for a neighbourhood graph this degree is essentially nil. The proof therefore stands as a piece of syntax while declaring, through the one rule it invokes, that the corpus underwrites it hardly at all: factivity is derivable but unsupported, the boxed content of a text descending to the text itself only on the vanishing set of self-accessible worlds. The same accounting applies verbatim to the derivations of $\mathsf{B},\mathsf{4},\mathsf{5}$ and $\mathsf{D}$, collected in the appendix, each of which turns on a single relational step --- $\mathrm{Sym}$, $\mathrm{Trans}$, $\mathrm{Eucl}$, $\mathrm{Ser}$ --- and so inherits, as its empirical weight, the corresponding entry of Table~\ref{tab:frame}, with its empirical uncertainty evaluated through the dependence-aware resampling procedure. A derivation in this calculus is thus never merely valid or invalid for the corpus; it is valid to a degree, the degree is read off the geometry of the neighbourhood graph, and the uncertainty surrounding that degree must be evaluated while accounting for the graph dependence structure. \S~\ref{subsec:tableaux} restates the same fact in a one-sided form in which the degree is nothing more, and nothing less, than a rate of set membership.

\subsection{A one-sided reformulation: sequent-style tableaux}\label{subsec:tableaux}

A two-sided sequent keeps apart, on the two flanks of the arrow, information that the licensing reading of Proposition~\ref{prop:licensing} never needs kept apart: every relational atom of $\mathsf{G3.K}$ lives on the antecedent alone, and $L\Box,R\Diamond$ merely repeat their principal formula there rather than moving it. The natural home of that reading is therefore not the two-sided sequent but the one-sided \emph{sequent-style tableaux} of Cuconato \cite{ref_cuconato_sttfol,ref_cuconato_sttpt}, a refutation calculus in which each node of the derivation tree carries a single finite \emph{block} $\Pi$ of formulas in place of a sequent. To a sequent $\Gamma\Rightarrow\Delta$ one associates the block $\Phi(\Gamma\Rightarrow\Delta)=\Gamma\cup\neg[\Delta]$, the antecedent together with the negations of the succedent formulas, so that the arrow is absorbed once and for all into a one-sided list of formulas supposed true; the rules of the cut-free sequent calculus are then read as decomposition rules on blocks, and the structural rules are absorbed into the set-theoretic reading of the block and into the closure condition on complementary pairs. Carried to the labelled modal language, the blocks hold labelled formulas $w{:}A$ and relational atoms $wRv$, a succedent formula $w{:}A$ entering the block as $w{:}\neg A$; under $\Phi$ the pairs $L\Box/R\Box$ and $L\Diamond/R\Diamond$ fuse into one persistent rule per operator, denoted $(\Box)$ and $(\Diamond)$, while $L\neg,R\neg$ leave no residue at all, the polarity of a formula being carried by the formula itself rather than by its side of the arrow. The relational rules pass to the single block unchanged:
\[
\frac{\Pi,\,wRv}{\Pi,\,wRv,\,vRw}\ (\mathrm{Sym})
\qquad
\frac{\Pi,\,wRv,\,vRu}{\Pi,\,wRv,\,vRu,\,wRu}\ (\mathrm{Trans})
\]
\[
\frac{\Pi,\,wRv,\,wRu}{\Pi,\,wRv,\,wRu,\,vRu}\ (\mathrm{Eucl})
\qquad
\frac{\Pi}{\Pi,\,wRv}\ (\mathrm{Ser})
\]
with $v$ fresh in $\mathrm{Ser}$, and with $\mathrm{Eucl}$ read, as in Definition~\ref{def:degree}, on its proper instances $v\neq u$, its diagonal instance $v=u$ adjoining the self-loop $vRv$.

That this transcription forfeits nothing is the content of the quotient theorem of Cuconato \cite{ref_cuconato2026tableaux}, which locates the labelled block calculus exactly against Negri's labelled sequent calculus. The map $\Phi$ is surjective, and its fibres consist of the sequents that differ only in the side on which a negated formula is displayed; along this quotient the rules of $\mathsf{G3.K}$ with primitive negation, together with the relational rules, descend to the block rules, every step of $L\neg$ projecting to the identity, every step of $R\neg$ to a double-negation step, and every remaining step to exactly one block step with the same branching and the same freshness conditions. Conversely every closed block derivation lifts to a two-sided derivation, so that a block derivation for $\Phi(\Gamma\Rightarrow\Delta)$ exists if and only if $\Gamma\Rightarrow\Delta$ is derivable in $\mathsf{G3.K}$, and the minimal block derivation is never larger than the minimal sequent one. The block calculus is thus $\mathsf{G3.K}$ seen through $\Phi$, modulo the bookkeeping of sides, and it inherits from Theorem~\ref{thm:structural} cut-freeness and the subformula property, now in the single language of blocks. This is the complete definition, and the complete justification, that the one-sided reading requires.

Call $\Pi_F:=\{\,wRv : w,v\in W,\ wRv\ \text{holds in}\ F\,\}$ the \emph{diagram} of the finite frame $F=(W,R)$: its complete positive relational record, prior to any valuation of atoms. An eligible ground instance of $\mathrm{Sym},\mathrm{Trans}$ or $\mathrm{Eucl}$ --- for $\mathrm{Eucl}$, a proper instance $wRv,wRu$ with $v\neq u$ --- is a tuple satisfying the rule's antecedent atoms in $\Pi_F$, and it is \emph{confirmed} if the atom the rule would add is already a member of $\Pi_F$, and \emph{speculative} otherwise; for $\mathrm{Ser}$, whose consequent is existential rather than a named atom, confirmation is read directly off $F$, a world counting as confirmed when it already possesses a recorded successor.

\begin{prop}[One-sided licensing]\label{prop:blocklicensing}
For $X\in\{\mathsf{B},\mathsf{4},\mathsf{5},\mathsf{D}\}$ with one-sided rule $\rho_X$, $\deg_X(F)$ is the proportion of eligible ground instances of $\rho_X$ on $\Pi_F$ that are confirmed. A derivation using one instance of $\rho_X$ is, in the sense of Proposition~\ref{prop:licensing}, licensed by $F$ exactly when that instance is confirmed, and $\deg_X(F)$ is the base rate at which $\rho_X$, applied to the diagram of $F$, is confirmed rather than speculative.
\end{prop}

\textbf{Proof.} An eligible instance of $\rho_X$ on $\Pi_F$ is, by construction of $\Pi_F$, exactly a tuple $\bar a$ with $F\models\varphi_X(\bar a)$, and it is confirmed exactly when its consequent atom lies in $\Pi_F$, that is, when $F\models\psi_X(\bar a)$; this is the same partition of tuples used in Definition~\ref{def:degree} and Proposition~\ref{prop:licensing}, read on one side of the sequent instead of two. $\blacksquare$

\begin{rem}
The one-sided reading is not mere economy of notation: it makes the binary
validation indicator literal rather than analogical. An eligible instance of
$\rho_X$ on $\Pi_F$ either is or is not confirmed, a membership test in the
single set $\Pi_F$, and $\deg_X(F)$ is the empirical frequency with which
that test succeeds over the eligible instances. This binary representation
does not make the instances statistically independent, since different rule
applications may share vertices, edges, or neighbourhoods. Accordingly, the
uncertainty surrounding group differences is evaluated through the
prompt-balanced resampling and influence analyses of
\S~\ref{subsec:estimator}, rather than through a binomial standard error.
In the open-world regime of \S~\ref{sec:discussion}, confirmation against
$\Pi_F$ is replaced by confirmation against a sampling distribution, and
$\deg_X$ is interpreted as a plug-in estimate of the population rate rather
than as the population rate itself.
\end{rem} 

The translation is most transparent on the derivation already before us. Negating the conclusion of factivity into the block $w{:}\Box A,\,w{:}\neg A$:
\begin{prooftree}
\AxiomC{$w{:}\Box A,\, w{:}\neg A \quad (\Pi)$}
\RightLabel{$(\mathrm{Ref})$}
\UnaryInfC{$w{:}\Box A,\, w{:}\neg A,\, wRw$}
\RightLabel{$(\Box)$}
\UnaryInfC{$\{\,w{:}\Box A,\, w{:}\neg A,\, wRw,\, w{:}A\,\} \times$}
\end{prooftree}
the block acquires $wRw$ by $(\mathrm{Ref})$, then $w{:}A$ by $(\Box)$, and closes on the complementary pair $\{w{:}A,w{:}\neg A\}$. The closure is exactly as good, and no better, than the confirmation of $wRw$ in $\Pi_F$, which by the seriality-style reading above occurs with frequency $\deg_{\mathsf{T}}(F)$ --- essentially never, for a $k$-nearest-neighbour diagram from which the self-edge is excluded by construction. The block forms of $\mathsf{B},\mathsf{4},\mathsf{5},\mathsf{D}$ are obtained from the two-sided derivations of the appendix by the same mechanical translation and are not reproduced here.

The gain of the one-sided form is not one of notation only, and it bears on the two questions this paper joins, the separation of the corpora and its computational reading. Three points deserve record. The block form makes the estimator picture of \S~\ref{subsec:estimator} literal rather than analogical: a relational instance is confirmed exactly when a single atom belongs to the single set $\Pi_F$, so that $\deg_X(F)$ is the frequency of a membership test, and the passage from the closed-world diagram to the open-world sampling distribution is the passage from testing membership in $\Pi_F$ to estimating a membership probability --- the very passage the resampling of \S~\ref{subsec:estimator} carries out, now with a syntactic object beneath it. Next, the block calculus is a calculus of refutation, saturating a single block in search of a countermodel exactly as the empirical procedure saturates a neighbourhood in search of a class boundary; and since the sign rules leave no residue, the number of saturation cases that every completeness and termination argument must check is halved with respect to the two-sided calculus \cite{ref_cuconato2026tableaux}, so that the cost of deciding a modal claim over the corpus is that of the labelled calculus and no more. Most to the present point, the relational vocabulary of the block calculus is indifferent to the shape of the atoms it records, and this indifference is decisive exactly at Euclideanity. Among the closure conditions, the Euclidean one is the only one whose proper instances relate two co-initial neighbours $v,u$ that the derivation leaves otherwise unordered --- siblings in its creation forest --- and a calculus that mirrors accessibility by the nesting of its syntax, as the tree-hypersequent and nested-sequent calculi do, cannot in general record such a sibling edge without changing its data structure, whereas the labelled block calculus spends on $\mathrm{Eucl}$ the same lines it spends on any other relational rule \cite{ref_cuconato2026tableaux}. Since Euclideanity is, with transitivity, one of the two axioms on which the human and artificial corpora separate most sharply, the format in which the separation is most naturally written is precisely the one that treats the discriminating axiom uniformly, and this is the sense in which the calculus is not merely an alternative notation but the apparatus fitted to the phenomenon.

\section{Discussion and conclusions}\label{sec:discussion}

The contribution of this paper is a single change of viewpoint, made precise twice. The frequencies with which a semantic neighbourhood graph is symmetric, transitive or Euclidean have been treated, in the literature that measures them, as descriptive geometry; we have shown that they are degrees of validation of the normal modal axioms $\mathsf{B},\mathsf{4},\mathsf{5},\mathsf{D}$, that through the labelled calculus of Negri each such degree is a syntactic quantity --- the proportion of the applications of a structural rule that the corpus licenses --- and that the same number is a statistical one, an empirical estimate of a licensing probability whose uncertainty must be evaluated through dependence-aware inference. The human and artificial corpora are thereby separated not merely by a
metaphor but by their position in a calculus. The empirical comparison is
strongest for the closure axioms $\mathsf{4}$ and $\mathsf{5}$: higher
artificial transitivity and Euclideanity persist after matching graph size
and prompt composition, removing individual prompt clusters, varying the
neighbourhood size, and replacing cosine distance with Euclidean distance on
normalised embeddings. Symmetry provides weaker and more
specification-sensitive evidence, while seriality is complete in both
$k$-nearest-neighbour subframes by construction. Beneath this, the situated relation $\Vdash_\sigma$, now graded by the degree of groundedness $g$ and the degree of situatedness $\mathrm{sit}$, records the same divide at the level of aboutness: by Proposition~\ref{prop:collapse} artificial, placeless texts ($g\to 0$) let the situational semantics collapse onto classical consequence, whereas place-grounded human texts keep the constraint of pertinence alive, the residual dispersion of their referents being measured by $\mathrm{sit}$.

The result invites two cautions, which are also its limits. First, the
resampling distributions should not be interpreted as arising from
independent relational tuples. The tuples defining symmetry, transitivity,
and Euclideanity overlap extensively within each graph. Our principal
uncertainty analysis therefore operates at the prompt level and reconstructs
equal-sized graphs by selecting one artificial response for each complete
prompt. The resulting percentile ranges quantify sensitivity to the observed
set of artificial responses conditional on the retained human corpus and
prompt design; they are not conventional confidence intervals for an
independent random sample of texts from an unrestricted population.
The leave-one-prompt-out jackknife serves a complementary purpose, assessing
whether any single thematic cluster disproportionately determines the
results. The closed-world to open-world passage is, in this language, the
passage from the observed plug-in estimate
$\widehat{\pi}_X=\deg_X(F)$ to a broader sampling distribution whose
interpretation depends on the mechanism by which prompts and texts are
generated. Second, the logic recovers only part of the data. The interior and boundary rates of Table~\ref{tab:boundary} translate into the modal conditions $\Box\alpha$ and $\Diamond\alpha\wedge\Diamond\neg\alpha$, and Proposition~\ref{prop:entropy} carries the local entropy along with them; but the global descriptors of Table~\ref{tab:global} and the centrality figures are metric invariants with no deductive counterpart, and we have left them outside the calculus deliberately rather than force an interpretation upon them.

\begin{figure}[htbp]
\centering
\resizebox{\textwidth}{!}{%
\begin{tikzpicture}[>={Stealth[length=5pt]}, line width=0.5pt, font=\small,
  box/.style={draw, rounded corners=2pt, align=center, inner sep=5pt},
  elab/.style={font=\scriptsize, fill=white, inner sep=1.5pt}]
  \node[box, text width=44mm] (geo) at (0,2.5)
    {\textsc{geometry}\\[1pt]\footnotesize $\mathcal{F}_k=(W,R_k)$: transformer $k$-NN frame\\ relational frequencies (Table~\ref{tab:frame})};
  \node[box, inner sep=5pt] (deg) at (0,0)
    {$\deg_X(F)\in[0,1]$\\[1pt]\footnotesize degree of validation};
  \node[box, text width=30mm] (pt) at (-4.9,0)
    {\textsc{proof theory}\\[1pt]\footnotesize $\mathsf{G3.K}+\rho_X$;\\ diagram $\Pi_F$;\\ sequent-style tableaux};
  \node[box, text width=30mm] (st) at (4.9,0)
    {\textsc{statistics}\\[1pt]\footnotesize $\widehat{\pi}_X=\deg_X(F)$;\\ prompt-balanced\\ estimate};
  \node[box, text width=62mm] (ai) at (0,-2.6)
    {\textsc{human vs.\ AI text}\\[1pt]\footnotesize closed artificial $\approx\mathsf{K45}$ \quad vs.\quad permeable human $\approx\mathsf{K}$};
  \draw[->] (geo) -- (deg) node[elab, midway, right=1pt] {frequency};
  \draw[<->] (pt) -- (deg) node[elab, midway, above=1pt] {licensing};
  \draw[<->] (st) -- (deg) node[elab, midway, above=1pt] {estimate};
  \draw[->] (deg) -- (ai) node[elab, midway, right=1pt] {separation};
\end{tikzpicture}%
}
\caption{The construction as a two-way dictionary. A single quantity, the degree of validation $\deg_X(F)$ of a modal axiom $X\in\{\mathsf{B},\mathsf{4},\mathsf{5},\mathsf{D}\}$ on the $k$-nearest-neighbour frame, is read at once as a relational frequency of the embedding geometry, as the licensing density of a structural rule $\rho_X$ in the labelled calculus $\mathsf{G3.K}$ and its one-sided block form on the diagram $\Pi_F$, and as a plug-in estimate $\widehat{\pi}_X$ of a population probability; the three readings coincide, and their common value separates the closed artificial regime near $\mathsf{K45}$ from the permeable human one near $\mathsf{K}$.}
\label{fig:dictionary}
\end{figure}
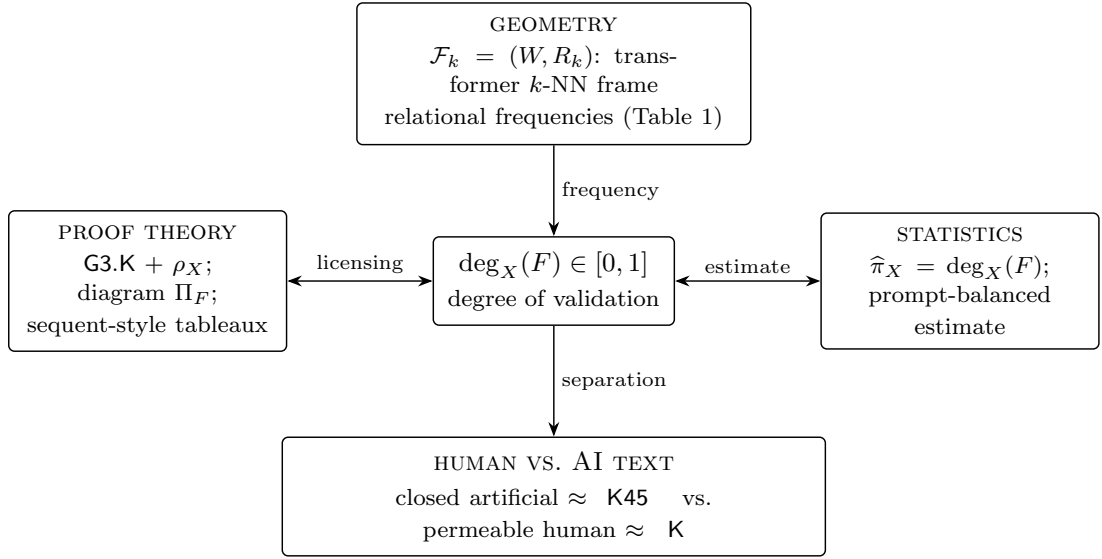

What the construction does deliver is a dictionary, exact in both directions, between a measurement on embeddings and a constraint on proofs, and between a constraint on proofs and a parameter to be estimated, as Figure~\ref{fig:dictionary} records. The natural continuation is twofold: to lift the dictionary from the closed-world to the open-world regime under alternative graph-sampling designs, and to give the degree of situatedness $\mathrm{sit}$ its own empirical estimator from the geography of georeferenced markers, so that aboutness --- and not only accessibility --- leaves a syntactic trace, and the silhouette, neighbourhood-preservation and within-class concentration of Table~\ref{tab:global} find their place not merely as substrate but as observable correlates of the lattice the reference function populates.

\section*{Data and code availability}
The transformer embeddings, the derived $k$-nearest-neighbour graphs, and the scripts that compute the relational frequencies, the prompt-balanced resampling, and the robustness analyses of Section~\ref{sec:degrees} are available from the authors on reasonable request.

\section*{Ethics and informed consent}
The human narratives analysed in this study were collected through an anonymous questionnaire. Participants were informed of the research purpose and consented to the use of their anonymised responses; no personal identifiers were retained in the analysed corpus, and the texts were processed solely in the aggregate form described in Section~\ref{sec:empirical}.

\appendix
\section{Derivations of the characteristic axioms}\label{app}

Each derivation below is correct in $\mathsf{G3.K}$ extended by the single relational rule named, and turns on one application of it; by Proposition~\ref{prop:licensing} it is licensed by a subframe to the degree recorded in the corresponding column of Table~\ref{tab:frame}, with the uncertainty surrounding that degree evaluated through the dependence-aware empirical analysis. The displayed leaves are not initial sequents in the strict sense --- an initial sequent requires an atomic $p$ --- but identity sequents $w{:}A,\Gamma\Rightarrow\Delta,w{:}A$ with $A$ arbitrary, which are height-preserving derivable in $\mathsf{G3.K}$ (Negri and von Plato \cite{ref_negrivonplato2011}); we display them as leaves to keep the single relational step in focus. Each derivation admits, by the translation of \S~\ref{subsec:tableaux}, a one-sided block equivalent turning on the same single confirmed or speculative atom; we display only the two-sided form here, since the passage between the two is mechanical and worked once, for $\mathsf{T}$, in \S~\ref{subsec:tableaux}.

\medskip
\noindent\emph{Symmetry, $\mathsf{B}=A\to\Box\Diamond A$, in $\mathsf{G3.K}+\mathrm{Sym}$.}
\begin{prooftree}
\AxiomC{$vRw,\ wRv,\ w{:}A \Rightarrow v{:}\Diamond A,\ w{:}A$}
\RightLabel{$R\Diamond$}
\UnaryInfC{$vRw,\ wRv,\ w{:}A \Rightarrow v{:}\Diamond A$}
\RightLabel{$\mathrm{Sym}$}
\UnaryInfC{$wRv,\ w{:}A \Rightarrow v{:}\Diamond A$}
\RightLabel{$R\Box$}
\UnaryInfC{$w{:}A \Rightarrow w{:}\Box\Diamond A$}
\RightLabel{$R{\to}$}
\UnaryInfC{$\Rightarrow w{:}A\to\Box\Diamond A$}
\end{prooftree}

\noindent\emph{Transitivity, $\mathsf{4}=\Box A\to\Box\Box A$, in $\mathsf{G3.K}+\mathrm{Trans}$.}
{\small
\begin{prooftree}
\AxiomC{$u{:}A,\ wRu,\ vRu,\ wRv,\ w{:}\Box A \Rightarrow u{:}A$}
\RightLabel{$L\Box$}
\UnaryInfC{$wRu,\ vRu,\ wRv,\ w{:}\Box A \Rightarrow u{:}A$}
\RightLabel{$\mathrm{Trans}$}
\UnaryInfC{$vRu,\ wRv,\ w{:}\Box A \Rightarrow u{:}A$}
\RightLabel{$R\Box$}
\UnaryInfC{$wRv,\ w{:}\Box A \Rightarrow v{:}\Box A$}
\RightLabel{$R\Box$}
\UnaryInfC{$w{:}\Box A \Rightarrow w{:}\Box\Box A$}
\RightLabel{$R{\to}$}
\UnaryInfC{$\Rightarrow w{:}\Box A\to\Box\Box A$}
\end{prooftree}}

\noindent\emph{Euclideanity, $\mathsf{5}=\Diamond A\to\Box\Diamond A$, in $\mathsf{G3.K}+\mathrm{Eucl}$.}
{\small
\begin{prooftree}
\AxiomC{$vRu,\ u{:}A,\ wRu,\ wRv,\ w{:}\Diamond A \Rightarrow v{:}\Diamond A,\ u{:}A$}
\RightLabel{$R\Diamond$}
\UnaryInfC{$vRu,\ u{:}A,\ wRu,\ wRv,\ w{:}\Diamond A \Rightarrow v{:}\Diamond A$}
\RightLabel{$\mathrm{Eucl}$}
\UnaryInfC{$u{:}A,\ wRu,\ wRv,\ w{:}\Diamond A \Rightarrow v{:}\Diamond A$}
\RightLabel{$L\Diamond$}
\UnaryInfC{$wRv,\ w{:}\Diamond A \Rightarrow v{:}\Diamond A$}
\RightLabel{$R\Box$}
\UnaryInfC{$w{:}\Diamond A \Rightarrow w{:}\Box\Diamond A$}
\RightLabel{$R{\to}$}
\UnaryInfC{$\Rightarrow w{:}\Diamond A\to\Box\Diamond A$}
\end{prooftree}}

\noindent\emph{Seriality, $\mathsf{D}=\Box A\to\Diamond A$, in $\mathsf{G3.K}+\mathrm{Ser}$.}
\begin{prooftree}
\AxiomC{$v{:}A,\ wRv,\ w{:}\Box A \Rightarrow w{:}\Diamond A,\ v{:}A$}
\RightLabel{$R\Diamond$}
\UnaryInfC{$v{:}A,\ wRv,\ w{:}\Box A \Rightarrow w{:}\Diamond A$}
\RightLabel{$L\Box$}
\UnaryInfC{$wRv,\ w{:}\Box A \Rightarrow w{:}\Diamond A$}
\RightLabel{$\mathrm{Ser}$}
\UnaryInfC{$w{:}\Box A \Rightarrow w{:}\Diamond A$}
\RightLabel{$R{\to}$}
\UnaryInfC{$\Rightarrow w{:}\Box A\to\Diamond A$}
\end{prooftree}

%=====================================================================
% BIBLIOGRAPHY
% This is the general bibliography you supplied, reproduced faithfully
% (only broken accent commands were repaired so the file compiles).
% Seven entries marked "% ADDED" were not present in the supplied list
% but are cited in the text and are indispensable (Negri's proof theory,
% the modal-correspondence references, and the statistics references).
% If you prefer only cited works to appear, migrate to a biblio.bib and
% \bibliography{biblio}; with thebibliography every listed entry prints.
%=====================================================================

\end{document}